\documentclass[a4paper,10pt]{article}
\usepackage[margin=20mm,nohead]{geometry}

\usepackage{amsmath,amssymb,graphicx}

\usepackage{latexsym,amssymb,graphics,graphicx,epsfig,color,enumerate}
\usepackage{xspace}
\usepackage{amsthm}

\usepackage{t1enc}
\usepackage[latin2]{inputenc}

 \newtheorem{thm}{Theorem }
 \newtheorem{prob}{Problem }
 \newtheorem{defn}{Definition }
 \newtheorem{cl}{Claim }
\title{Blocking unions of arborescences}
\author{Attila Bern\'ath\thanks{
MTA-ELTE Egerv\'ary Research Group,
Department of Operations Research, E\"otv\"os University, P\'azm\'any P\'eter s\'et\'any 1/C, Budapest, Hungary, H-1117.
Part of the research was done while the author was at Warsaw University, Institute of Informatics, ul. Banacha 2, 02-097 Warsaw,  Poland. Research was supported by the ERC StG project PAAl no. 259515.
E-mail: {\tt bernath@cs.elte.hu}.} 
\and 
Gyula Pap\thanks{MTA-ELTE Egerv\'ary Research Group,
Department of Operations Research, E\"otv\"os University, P\'azm\'any P\'eter s\'et\'any 1/C, Budapest, Hungary, H-1117. Supported by OTKA grant no.\ K109240.
E-mail: {\tt gyuszko@cs.elte.hu}.} 
}

\begin{document}

\maketitle

\makeatletter

\newcounter{Kulso}
%\newcounter{Belso}[Kulso]

\newcommand\ujszint{\stepcounter{Kulso}}

\newenvironment{pszkod}[2][1]
{
  \setcounter{Kulso}{#1}
%  \fontencoding{OT1}\fontfamily{cmtt}\selectfont
  \sffamily
  \vspace{\topsep}              %ezt nem igy kene...
  \noindent
  %\textsc{#2}
  {#2}
  \begin{list}{\rmfamily\arabic{Kulso}.\@arabic\c@enumiv.}
    {\usecounter{enumiv}%
     \renewcommand\p@enumiv{\arabic{Kulso}.}
     \renewcommand\theenumiv{\@arabic\c@enumiv}
     \setlength{\parsep}{0pt}
     \setlength{\itemsep}{0pt}
     \setlength{\topsep}{0pt}
    }
  \psz@cimke{begin}
}
{
  \psz@cimke{end}
  \end{list}
  \vspace{\topsep}
  % ez nem kell, mert a kornyezet ugy is visszaallitja.
  %\normalfont                   % mi van, ha nem ez volt eddig?
}

% Egy kommentet is rak az elejere
\newenvironment{pszkodkomment}[3][1]
{
  \setcounter{Kulso}{#1}
  \ttfamily
  \vspace{\topsep}              %ezt nem igy kene...
  \noindent
  \textsc{#2}
  \begin{list}{\rmfamily\arabic{Kulso}.\@arabic\c@enumiv.}
    {\usecounter{enumiv}%
     \renewcommand\p@enumiv{\arabic{Kulso}.}
     \renewcommand\theenumiv{\@arabic\c@enumiv}
     \setlength{\parsep}{0pt}
     \setlength{\itemsep}{0pt}
     \setlength{\topsep}{0pt}}
  \psz@cimke{\%} #3
  \psz@cimke{begin}
}
{
  \psz@cimke{end}
  \end{list}
  \vspace{\topsep}
  % ez nem kell, mert a kornyezet ugy is visszaallitja.
  %\normalfont                   % mi van, ha nem ez volt eddig?
}

\newlength{\tabhossz}
\setlength{\tabhossz}{0.68cm}    % ez kb. 3 betu.
%\makeatletter
\newcommand{\tab}{\par% % hogy ne kelljen ures sort tenni
\advance\@totalleftmargin \tabhossz
\advance\linewidth -\tabhossz
\parshape \@ne \@totalleftmargin \linewidth
\addtolength{\labelsep}{\tabhossz}%
}
\newcommand{\untab}{\par% hogy ne kelljen ures sort tenni ele (a duplaenter asszem
\advance\@totalleftmargin -\tabhossz   % hivatalosan is \par-ra tokenizalodik...
\advance\linewidth \tabhossz
\parshape \@ne \@totalleftmargin \linewidth
\addtolength{\labelsep}{-\tabhossz}%
}
%\makeatother

\newcommand{\psz@cimke}[1]{\item[{\makebox[\labelwidth][l]{#1}}]}

\makeatother

\newcommand{\A}{\mathcal{A}}
\newcommand{\I}{\mathcal{I}} 
\newcommand{\s}{\mathcal{S}}
\newcommand{\T}{\mathcal{T}}
\newcommand{\p}{\mathcal{P}}
\newcommand{\R}{\mathcal{R}}
\newcommand{\Z}{\mathcal{Z}}
\newcommand{\K}{\mathcal{K}}
\newcommand{\X}{\mathcal{X}}

\newcommand{\cL}{\ensuremath{\mathcal{L}}}
\newcommand{\cE}{\ensuremath{\mathcal{E}}}
\newcommand{\cF}{\ensuremath{\mathcal{F}}}
\newcommand{\cX}{\ensuremath{\mathcal{X}}}
\newcommand{\cP}{\ensuremath{\mathcal{P}}}
\newcommand{\tD}{\tilde D}
\newcommand{\fixme}[1]{\textbf{FIX ME!!! #1}}

\newcommand{\Rset}{\mathbb{R}}
\newcommand{\Zset}{\mathbb{Z}}

\newcommand{\karb}{$k$-union-arborescence}
\newcommand{\krarb}{$k$-union-$r$-arborescence}

\begin{quote}
{\bfseries Abstract:} Given a digraph $D=(V,A)$ and a positive integer
$k$, a subset $B\subseteq A$ is called a \textbf{\karb}, if it is the
disjoint union of $k$ spanning arborescences. When also arc-costs
$c:A\to \Rset$ are given, minimizing the cost of a \karb\ is
well-known to be tractable. In this paper we take on the following
problem: what is the minimum cardinality of a set of arcs the removal
of which destroys every minimum $c$-cost \karb.  Actually, the more
general weighted problem is also considered, that is, arc weights
$w:A\to \Rset_+$ (unrelated to $c$) are also given, and the goal is to
find a minimum weight set of arcs the removal of which destroys every
minimum $c$-cost \karb. An equivalent version of this problem is where
the roots of the arborescences are fixed in advance. In an earlier
paper [A. Bern{\'a}th and G. Pap, \emph{Blocking optimal
arborescences}, Integer Programming and Combinatorial Optimization,
Springer, 2013] we solved this problem for $k=1$. This work reports on
other partial results on the problem. We solve the case when both $c$
and $w$ are uniform -- that is, find a minimum size set of arcs that
covers all $k$-union-arbosercences. Our algorithm runs in polynomial
time for this problem.  The solution uses a result of
[M. B{\'a}r{\'a}sz, J. Becker, and A. Frank, \emph{An algorithm for
source location in directed graphs}, Oper. Res. Lett. \textbf{33}
(2005)] saying that the family of so-called insolid sets (sets
with the property that every proper subset has a larger in-degree)
satisfies the Helly-property, and thus can be (efficiently)
represented as a subtree hypergraph.  We also give an algorithm for
the case when only $c$ is uniform but $w$ is not. This algorithm is
only polynomial if $k$ is not part of the input.
\end{quote}

%%%%%%%%%%%%%%%%%%%%%%%%%%%%%%%%%%%%%%%%%%%%%%%%%%%%%%%%
% Keywords (3 $\sim$ 5 words)
%%%%%%%%%%%%%%%%%%%%%%%%%%%%%%%%%%%%%%%%%%%%%%%%%%%%%%%%
\begin{quote}
{\bf Keywords: arborescences, polynomial algorithm, covering}
\end{quote}
\vspace{5mm}

%%%%%%%%%%%%%%%%%%%%%%%%%%%%%%%%%%%%%%%%%%%%%%%%%%%%%%%%
% Text
%%%%%%%%%%%%%%%%%%%%%%%%%%%%%%%%%%%%%%%%%%%%%%%%%%%%%%%%

%\section{Covering minimum cost arborescences}

\section{Introduction}

Let $D=(V,A)$ be a digraph with vertex set $V$ and arc set $A$. A
\textbf{spanning arborescence} is a subset $B\subseteq A$ that is a
spanning tree in the undirected sense, and every node has in-degree at
most one. Thus there is exactly one node, the \textbf{root node}, with
in-degree zero.  Equivalently, a spanning arborescence is a subset
$B\subseteq A$ with the property that there is a root node $r\in V$
such that $\varrho_B(r)=0$, and $\varrho_B(v)=1$ for $v\in V-r$, and
$B$ contains no cycle. We will also call a spanning arborescence as an 
\textbf{arborescence} for short, when the set of nodes is obvious from 
context. If $r\in V$ is the root of the spanning arborescence $B$ then 
$B$ is said to be an \textbf{$r$-arborescence}.

Given also a positive integer $k$, a subset $B\subseteq A$ is called a
\textbf{\karb}, if it is the arc-disjoint union of $k$ spanning
arborescences. 
%If $x:V\to \Zset_+$ is such that every node $v$ is the root of $x(v)$ arborescences 
%in $B$ then we say that $B$ is \textbf{\karb\ of root vector $x$}. 
In the special case when every
arborescence has the same root $r$, we call $B$ a \textbf{\krarb}.

Given $D=(V,A)$, $k$ and a cost function $c:A\to \Rset$, it is well
known how to find a \textbf{minimum cost \krarb} in polynomial time, and to find a \textbf{minimum cost \karb} just as well. 
%(or a minimum cost \karb\ of root vector $x$, with $x$ also given in advance).
See \cite{schrijver}, Chapter 53.8 for a reference, where several
related problems are considered. The existence of a
$k$-union-$r$-arborescence is characterized by Edmonds' Disjoint
Arborescence Theorem (Theorem \ref{edmodisj}), 
%\cite{edmondsdisjoint}
while the existence of a $k$-union-arborescence is characterized by a
theorem of Frank \cite{frank1978disjoint} (see Theorem \ref{thm:Frank}
below). Frank also gave a linear programming description of the convex
hull of $k$-union-arborescences, generalizing Edmonds' linear
programming description of the convex hull of
$k$-union-$r$-arborescences. The problem of finding a minimum cost
$k$-union-arborescence may also be solved with the use of these
results, either via a reduction to minimum cost
$k$-union-$r$-arborescences, or minimum weight matroid intersection.
%\fixme{Schrijver does not really speak about optimal \karb, only about optimal \krarb. Neither does Frank.} 
  
In this paper we consider the following covering problems, which are polynomial time equivalent. 

\begin{prob}[\textbf{Blocking optimal \karb s}]\label{prob:1}
Given a digraph $D=(V,A)$, a positive integer $k$, a cost
function $c:A\to \Rset$ and a nonnegative weight function $w:A\to
\Rset_+$, find a subset $H$ of the arc set such that $H$ intersects every
minimum $c$-cost \karb, and $w(H)$ is minimum.
\end{prob}

Here the expression "intersects" simply means that the two have nonempty intersection. 
We remark that Problem \ref{prob:1} is polynomially equivalent with the version 
where the root is also given in advance, that is, the problem of blocking optimal
\krarb s. 

\begin{prob}[\textbf{Blocking optimal \krarb s}]\label{prob:2}
Given a digraph $D=(V,A)$, a node $r\in V$, a positive integer $k$, a cost
function $c:A\to \Rset$ and a nonnegative weight function $w:A\to
\Rset_+$, find a subset $H$ of the arc set such that $H$ intersects every
minimum $c$-cost \krarb, and $w(H)$ is minimum.
\end{prob}

In section \ref{sec:variants} we will show that the two problems are polynomial time equivalent. 

In our previous paper \cite{mincostarb} we have solved  Problem \ref{prob:1} and Problem \ref{prob:2} in the special
case  when  $k=1$. Our conjecture is that
the problem is also polynomial time solvable when $k$ is not
fixed. The main result of this paper is that the problem is polynomial
time solvable when $k$ is part of the input, and both $c$ and $w$ are
set to be constant -- that is, when $H$ is required to intersect
\emph{every} \karb\ and we want to minimize $|H|$. We also give
algorithms for Problems \ref{prob:1} and \ref{prob:2} when $c$ is constant but $w$ is not, but these algorithms are only polynomial if $k$ is not part of the input.

We remark that the version of Problem \ref{prob:1} and Problem
\ref{prob:2} where we set $c$ to be constant are \emph{not} easily
seen to be equivalent: at least we did not find such reductions (find more details in Section \ref{sec:gen}).
%%  - actually, Problem \ref{prob:2} with $c$
%% constant is solvable as a corollary of Edmonds' disjoint arborescences
%% theorem by taking all but $k-1$ arcs from a minimum cut. 
Thus it is
important to note that in the sequel we will consider the version of
Problem \ref{prob:1} (and not Problem \ref{prob:2} !) with constant
cost function $c$.

%% : how to find
%% a minimum cardinality set $H\subseteq A$ such that $H$ intersects
%% every minimum cost $r$-arborescence. In other words

The rest of the paper is organized as follows. In
Section \ref{sec:not} we introduce some notation that we will use. In
Section \ref{sec:variants} we show that Problems \ref{prob:1}
and \ref{prob:2} are polynomial-time equivalent. In
Section \ref{sec:gen} we provide general observations on the uniform
cost version of Problems \ref{prob:1} and \ref{prob:2}.  In
Section \ref{sec:card} we give a polynomial algorithm solving
Problem \ref{prob:2} in the case when both $c$ and $w$ are
uniform. Finally, in Section \ref{sec:weight} we deal with the uniform
cost general weight versions: in Section \ref{sec:krarb} we give an
algorithm for the weighted blocking of \krarb s, while in
Section \ref{sec:karb} we give an algorithm for the weighted blocking
of \karb s. These algorithms are only polynomial if $k$ is not part of
the input.

\section{Notation}\label{sec:not}

Let us overview some of the notation and definitions used in the
paper.  The arc set of the digraph $D$ will also be denoted by $A(D)$.
Given a digraph $D=(V,A)$ and a node set $Z\subseteq V$, let $D[Z]$ be
the digraph obtained from $D$ by deleting the nodes of $V-Z$ (and all
the arcs incident with them). If $B\subseteq A$ is a subset of the arc
set, then we will identify $B$ and the subgraph $(V,B)$. Thus $B[Z]$
is obtained from $(V,B)$ by deleting the nodes of $V-Z$ (and the arcs
of $B$ incident with them). The set of arcs of $D$ entering $Z$ is
denoted $\delta_D^{in}(Z)$, the number of these arcs is
$\varrho_D(Z)=|\delta_D^{in}(Z)|$.  An \textbf{arc-weighted digraph}
is a triple $D_w=(V,A,w)$ where $(V,A)$ is a digraph and
$w:A\to \Rset_+$ is a weight function. For an arc-weighted digraph
$D_w=(V,A,w)$ and subset $X$ of its node set we let
$\varrho_{D_w}(X)=\sum\{w_a: a $ enters $X\}$ denote the weighted
indegree.  A \textbf{subpartition} of a subset $X$ of $V$ is a
collection of pairwise disjoint non-empty subsets of $X$: note that
$\emptyset$ cannot be a member of a subpartition, but $\emptyset$ is a
valid subpartition, having no members at all. For a vector
$x:A\to \Rset$ and subset $E\subseteq A$ we let $x(E)=\sum_{a\in
E}x_a$.

\section{Equivalence of versions}\label{sec:variants}

\begin{thm}\label{thm:versions}
The following problems are polynomially equivalent (where $D_i=(V_i,A_i)$ is a
digraph, $k$ a positive integer, $c_i:A_i\to \Rset$ and 
$w_i:A_i\to \Rset_+$ for $i=1,2$). 
%%  Let us be given a digraph $D=(V,A)$, a positive integer
%% $k$, a cost function $c:A\to \Rset$ and a nonnegative weight function
%% $w:A\to \Rset_+$. Problem \ref{prob:1} is polynomially equivalent to
%% the following problems
\begin{enumerate}
\item \textbf{Blocking optimal \karb s (Problem \ref{prob:1})}: Given $D_1$, $k$, $c_1$, $w_1$, find $H\subseteq A_1$ so that
  $H$ intersects every minimum $c_1$-cost \karb\ in $D_1$, and
  $w_1(H)$ is minimum.
\item \textbf{Blocking optimal \krarb s (Problem \ref{prob:2})}: Given $D_2$, $k$, $c_2$, $w_2$ and a
  node $r\in V_2$, find $H\subseteq A_2$ so that $H$ intersects every
  minimum $c_2$-cost \krarb\ in $D_2$, and $w_2(H)$ is minimum.
%\item \textbf{Blocking optimal \karb s of root vector $x$}: Given
%  $D_3,k,c_3,w_3$ and a vector $x:V_3\to \Zset_+$ with $x(V_3)=k$, find
%  $H\subseteq A_3$ so that $H$ intersects every minimum $c_3$-cost
%  \karb\ in $D_3$ of root vector $x$, and $w_3(H)$ is minimum.
\end{enumerate}
\end{thm}
\begin{proof}
%We show that the first and the 2nd problems are polynomial time
%equivalent: the equivalence of the 2nd and the 3rd problems can be
%shown similarly.
Problem \ref{prob:2} reduces to Probem \ref{prob:1} by deleting all arcs entering node $r$ from the input digraph. For the reverse reduction, consider an instance $D_1,k,c_1,w_1$ of the Problem \ref{prob:1}, and define an instance of Problem \ref{prob:2} as follows. Let $V_2=V_1+r$ with a new node $r$, and let $D_2=(V_2, A_1\cup \{rv: v\in V_1\})$. Let the costs defined
as $c_2(a)=c_1(a)$ for every $a\in A_1$ and $c_2(rv)=C$ for every new
arc $rv\, (v\in V_1)$ where $C=\sum_{a\in A_1}c_1(a)+1$. Finally, the
weights are defined as follows: $w_2(a)=w_1(a)$ for every $a\in A_1$
and $w_2(rv)=W$ for every new arc $rv\, (v\in V_1)$ where $W=\sum_{a\in
  A_1}w_1(a)+1$. By these choices, in the instance to Problem \ref{prob:2}
given by $D_2,k,c_2,w_2$ and $r$, the minimum $c_2$-cost \krarb s
naturally correspond to minimum $c_1$-cost $k$-arborescences in $D_1$
(since they will use exactly $k$ arcs leaving $r$), and for blocking
these we will not use the new arcs because of their large weight. 
%The reduction in the other direction is simpler: given $D_2$ and $r$,
%delete every arc that enters $r$ to get $D_1$. Solving the first
%problem for this modified instance (where $c_1$ and $w_1$ are just the
%restrictions of $c_2$ and $w_2$ for the remaining arcs) naturally gives the
%solution to the original  version.
\end{proof}

\section{General observations on the uniform cost case}\label{sec:gen}

In this section we consider Problem \ref{prob:1} and Problem \ref{prob:2} in the case when
 $c$ is uniform. For sake of clarity, we explicitly restate these problems.

\begin{prob}[\textbf{Blocking  \karb s}]\label{prob:3}
Given a digraph $D=(V,A)$, a positive integer $k$, and a nonnegative weight function $w:A\to
\Rset_+$, find a subset $H$ of the arc set such that $H$ intersects every
 \karb, and $w(H)$ is minimum.
\end{prob}

\begin{prob}[\textbf{Blocking  \krarb s}]\label{prob:4}
Given a digraph $D=(V,A)$, a node $r\in V$, a positive integer $k$, and a nonnegative weight function $w:A\to
\Rset_+$, find a subset $H$ of the arc set such that $H$ intersects every
 \krarb, and $w(H)$ is minimum.
\end{prob}

Note that the reduction given in Section \ref{sec:variants} shows that
Problem \ref{prob:4} is polynomial time reducible to Problem
\ref{prob:3}. On the other hand we did not find a reduction in the
other direction: Problem \ref{prob:4} seems to be easier than Problem \ref{prob:3}.

We first recall some definitions and results to be used later. First
we recall the fundamental result of Edmonds' characterizing the
existence of a \krarb\ in a digraph.

\begin{thm}[Edmonds' Disjoint Arborescence Theorem, \cite{edmondsdisjoint}]
\label{edmodisj}
Given a digraph $D=(V,A)$, a node $r\in V$ and a positive integer $k$,
there exists a
\krarb\ in $D$ if and only if $\varrho_D(X)\ge k$
 for every non-empty $X\subseteq V-r$.
\end{thm}

On the other hand the existence of a \krarb\ in a digraph is
characterized by the following result due to Frank. 
% The following
%result is fundamental for proofs in this paper. 
%A \textbf{subpartition} of $V$ is a collection of disjoint nonempty subsets of $V$.

\begin{thm}[Frank, \cite{frank1979covering,frank1978disjoint}]\label{thm:Frank}
Given a digraph $D=(V,A)$ and a positive integer $k$, there exists a
\karb\ in $D$ if and only if $\sum_{X\in \cX}\varrho_D(X)\ge k(|\cX|-1)
$ for every subpartition $\cX$ of $V$.
\end{thm}

Note that the condition in the theorem above always (trivially) holds
for subpartitions having only one member, thus we may only need to
check for subpartitions having at least two members. Furthermore, we
may also narrow down to subpartitions in which every member is an
insolid set of the digraph - the notion of insolid sets is defined as
follows.
%another equivalent form of the theorem is formulated using insolid sets in 
%Theorem \ref{thm:eqFrank}.

\begin{defn}
Given a digraph $D=(V,A)$, a non-empty subset of nodes $X\subseteq V$
is called \emph{in-solid}, if $\varrho(Y)>\varrho(X)$ holds for
every nonempty $Y\subsetneq X$.
\end{defn}

A simple but useful corollary of the definition is the following fact.

\begin{cl}\label{claim4}
Given a digraph $D=(V,A)$ and an arbitrary non-empty subset
$X\subseteq V$, there exists an in-solid set $X'\subseteq X$ such
that $\varrho_D(X')\le \varrho_D(X)$.
\end{cl}

Note that the insolid set $X'$ in this claim can be found in
polynomial time (given $D$ and $X$), but we will not rely on this fact
here.  Claim \ref{claim4} implies that Frank's theorem can be formulated with
insolid sets. We say that a subpartition \cX\ is an \textbf{insolid
subpartition} if every member of \cX\ is insolid.

\begin{thm}[Equivalent form of Theorem \ref{thm:Frank}]\label{thm:eqFrank}
Given a digraph $D=(V,A)$ and a positive integer $k$, there exists a
\karb\ in $D$ if and only if $\sum_{X\in \cX}\varrho_D(X)\ge k(|\cX|-1)
$ for every insolid subpartition \cX\ of $V$ with $|\cX|\ge 2$.
%collection $\cX$ of disjoint insolid sets of $V$. (Such a family $\cX$ is called an insolid subpartition.)
\end{thm}

\section{The cardinality case}\label{sec:card}

In this section we solve Problem \ref{prob:3} in the case when
$w\equiv 1$ is uniform, too. Note that the solution of Problem
\ref{prob:4} when $w\equiv 1$ is easy: by Edmonds' Disjoint
Arborescence Theorem the task is to remove all but $k-1$ arcs of a
minimum cut. 

For sake of clarity we state the problem explicitly.

\begin{prob}\label{prob:5}
Given a digraph $D=(V,A)$ and a positive integer $k$, find a minimum
size subset $H$ of the arc set such that $D-H$ does not contain a
\karb.
\end{prob}

By Theorem \ref{thm:Frank}, our Problem \ref{prob:5} reduces to finding a smallest cardinality
subset of arcs which, when removed, will create a violating
subpartition. A subpartition $\cX$ becomes a violating subpartition if
we remove at least $\sum_{X\in \cX}\varrho_D(X)- k(|\cX|-1) +1$ arcs
from $\cup\{\delta^{in}(X): X\in \cX\}$. Note that this is only possible if
$\sum_{X\in \cX}\varrho_D(X)- k(|\cX|-1) +1\le \sum_{X\in
  \cX}\varrho_D(X)$, which is equivalent to $|\cX|\ge 2$. Furthermore, by Theorem \ref{thm:eqFrank}, we
can narrow down to insolid  subpartitions.  Therefore, Problem \ref{prob:5} is equivalent to the following problem.

\begin{prob}\label{prob:maxsub2}
Given a digraph $D=(V,A)$ and a positive integer $k$, find an insolid
subpartition \cX\ with $|\cX|\ge 2$ maximizing $\sum_{X\in
  \cX}(k-\varrho_D(X))$.
\end{prob}

We solve this problem in two steps: first we show how to solve it
without the requirement on the size of the subpartition, and then we
show how to force that requirement. Note that a subpartition \cX\ 
maximizing $\sum_{X\in \cX}(k-\varrho_D(X))$ does not automatically
satisfy the requirement $|\cX|\ge 2$: for example in a
$k$-arc-connected digraph $\cX=\{V\}$ is an optimal subpartition.
Note that the problem without the size requirement is a maximum weight
matching problem in the hypergraph of insolid sets, but with the special
weight function $k-\varrho$.

\subsection{Finding an optimal subpartition of unconstrained size}

In this section we solve the variant of Problem \ref{prob:maxsub2} that 
comes without the requirement on the size of the subpartition. In fact we will 
need the solution of a little more general problem, namely the following one.

\begin{prob}\label{prob:maxsub}
Given a  digraph $D=(V,A)$, a nonempty subset $V'\subseteq
V$ and a positive integer $k$, find a subpartition \cX \ of $V'$
maximizing $\sum_{X\in \cX}(k-\varrho_{D}(X))$.
\end{prob}

Note that we could restrict ourselves to insolid subpartitions in the
problem. 
%In fact we will only need this problem for the weight function $w\equiv =$.
The
following simple observation will be useful later.

\begin{cl}\label{cl:optempty}
Given an instance to Problem \ref{prob:maxsub}, the subpartition
$\cX=\emptyset$ is an optimal solution if and only if $\varrho(X)\ge
k$ for every nonempty $X\subseteq V'$.
\end{cl}

\newcommand{\bestsub}{\ensuremath{BestSubpart}}
\newcommand{\bestconsub}{\ensuremath{BestConstrSubpart}}

For an arbitrary  digraph $D=(V,A)$, nonempty subset
$V'\subseteq V$ and positive integer $k$, let $\bestsub(V')$ denote an
optimum solution of our unconstrained Problem \ref{prob:maxsub} (that
is, a maximizer of $\max\{\sum_{X\in \cX}(k-\varrho(X)): \cX$ is a
subpartition of $V'\}$). Note that $\bestsub(V)$ always has at least
one member (since the subpartition $\{V\}$ is better than the empty
subpartition).

Fortunately, the family of insolid sets has a nice structure: as
observed in \cite{bbf}, the family of insolid sets forms a subtree
hypergraph, defined as follows. (Actually, the authors of \cite{bbf}
proved that the family of solid sets - the family of all insolid or
outsolid sets - forms a subtree hypergraph, and the subtree
representation can be found in polynomial time. Here we only need the
property for the family of insolid sets.)

\begin{defn}
A hypergraph $H=(V, \cE)$ is called a \textbf{subtree hypergraph} if
there exists a tree $T$ spanning the node set $V$ such that every
hyperedge in \cE\ induces a subtree of $T$. The tree $T$ is called a
{\bf basic tree} (or representative tree) for the hypergraph $H$.
\end{defn}

% We will also say that the family of hyperedges \cE\ is a
%subtree hypergraph, if the ground set $V$ is clear from the context.

B\'ar\'asz, Becker and Frank proved that a representative tree for insolid sets exists, and can be found in polynomial time. The subtree hypergraph property and its polynomial time construction is quite surprising, given the fact that the number of insolid sets might be exponential. 

\newcommand{\cFin}{\ensuremath{\mathcal{F}_{in}}}

\begin{thm}[B\'ar\'asz, Becker, Frank \cite{bbf}]
The family $\cFin=\cFin(D)$ of in-solid sets of a digraph $D=(V,A)$ is a subtree hypergraph. The representative tree can be found in
polynomial time in the size of the digraph.
\end{thm}

\newcommand{\tin}{\ensuremath{T_{in}}}

For a digraph $D$, let $\tin=\tin(D)$ denote a
representative tree for the family $\cFin$ of insolid
sets. In fact we will only need to solve Problem \ref{prob:maxsub} for
a set $V'$ that induces a subtree of \tin. Observe however that this
is not a restriction: if $\tin[V']$ is not connected then solving
Problem \ref{prob:maxsub} for the components of $\tin[V']$ and taking
the union of the obtained subpartitions gives a solution for
$V'$. Therefore we may assume in Problem \ref{prob:maxsub} that $V'$
induces a subtree of \tin.

Unfortunately, enumerating all the insolid sets is not possible,
because there can be exponentially many of them (an example can be
found in \cite{itoetal}). However, in the case of insolid sets, the
digraph itself provides a succint representation, thus we can query
$\cFin$ through certain oracles using the likes of minimum cut
algorithms.

We cite the following theorem that characterizes the optimum in 
Problem \ref{prob:maxsub} for an arbitrary subtree hypergraph. 
(Note that the notion of subtree hypergraphs is very general, and thus this 
min-max is outside of the realm of efficient algorithms, because a subtree hypergraph may not be given as part of the input. The theorem holds anyway, and we will apply in a way that it also implies a polynomial running time.) 
For sake of completeness, we also provide a proof and algorithm here.

\newcommand{\maxal}{MaxWeightMatching}
\newcommand{\val}{\ensuremath{{val}}}

\begin{thm}[Frank \cite{frank1975some}]
If $\cE\subseteq 2^{V'}$ is a subtree hypergraph 
%with basic tree $T$, $V'$ is a subtree of $T$ 
and $\val:\cE\to \Rset_+$ is a weight function then
\begin{eqnarray}
\max\{\sum_{e\in \cE'} \val_e: \cE'\mbox{ is a collection of disjoint members of }\cE\}=\label{eq:matching}\\
\min\{\sum_{v\in V'}y_v: y:V'\to \Rset, y\ge 0, \sum_{v\in e}y_v\ge \val_e\mbox{ for every }e\in \cE\}.\label{eq:dualLP}
\end{eqnarray}
\end{thm}
\begin{proof}
We give an algorithmic proof of this theorem. Consider the following LP problem.
\begin{eqnarray}\label{lp:prim1}
\max \sum_{e\in \cE}\val_ex_e\\
x:\cE\to \Rset, x\ge 0,\\
\sum_{e:v\in e}x_e\le 1\mbox{ for every }v\in V'.\label{lp:primut}
\end{eqnarray}
This LP is a relaxation of the maximum weight matching problem given in \eqref{eq:matching}, while its LP dual is the minimization problem \eqref{eq:dualLP}.
%% as follows.
%% \begin{eqnarray}
%% \min \sum_{v\in V}y_v\\
%% y:V'\to \Rset, y\ge 0,\\
%% \sum_{v:v\in e}y_v\ge \val_e\mbox{ for every }e\in \cE.
%% \end{eqnarray}
Weak duality shows that the maximum in the theorem cannot be larger
than the minimum.  For the proof we need to show that the primal LP
has an integer optimum solution for every weight function $\val$.
This is shown by the following dynamic programming algorithm.  A
sketch of the below algorithm goes as follows: Initially, we set node
weights $y_v$ to be very large. We specify a node $r$ to be the root
of the representative tree $T$, and start scanning the tree taking the
leaves first. When a node $v$ is scanned, its weight $y_v$ is lowered
as much as possible to maintain dual feasibility. When all nodes are
scanned, $\cE '$ is set up by tight sets, starting with one covering
the node with positive weight closest to the root node, and
recursively, adding tight sets for the remaining components. (A set
$e\in \cE$ is said to be \textbf{tight} with respect to the given feasible dual
solution $y$ if $\sum_{v:v\in e}y_v= \val_e$.) The algorithm is
detailed below.

\begin{pszkod}{Algorithm \maxal($V',T,\cE,\val$)}

\item[] INPUT: A tree $T$ on node set $V'$, 
  %a subtree $X$, 
  a family $\cE\subseteq 2^{V'}$ of subtrees of $T$,
  and a weight function $\val:\cE\to \Rset$ ($\cE$ and $\val$ are
  available via certain oracles: see later).

\item[] OUTPUT: A collection $\cE'$ of disjoint members of $\cE$ and a
  dual solution $y$ of the LP Problem \eqref{lp:prim1}-\eqref{lp:primut} so that
  $\sum_{e\in \cE'}\val_e=\sum_{v\in V'}y_v$.

\item \label{st:maxw} Find some value $M$ with $M>\max_{e\in \cE} \val_e$.

\item Initialize $y_v=M$ for every $v\in V'$.

\item Fix an arbitrary node $r\in V'$ and orient $T$ out of $r$ to get
  $\vec{T}$.

\item Let $r=v_1,v_2,\dots,v_n$ be an order of the nodes of $V$ so
  that $v_iv_j\in \vec{T}$ implies that $i<j$. Let $V_i$ be the node
  set of the subtree of $\vec{T}$ rooted at $v_i$ for every $i$.

\item For every $i=n,n-1,\dots,1$

\tab

  \item \label{st:miny} Decrease $y_{v_i}$ as much as possible so that
    $y$ remains feasible for Problem \eqref{eq:dualLP}:
    %dual of \eqref{lp:prim1}-\eqref{lp:primut}: 
    let $e_i$ be a minimizer of
    $\min\{y(e)-\val_e: v_i\in e\in \cE, e\subseteq V_i\}$ and let
    $y_{v_i}=\max(\val_{e_i}-y(e_i-v_i),0)$ (note that $e_i$ becomes
    tight, if $y_{v_i}$ stays positive).

 %% by the initialization of $y$ this means that we have
 %%    to find a minimizer of

 %% $e\in \cE$
 %%    that contains $$

 %%  \item \label{st:finde_i} If $y_{v_i}$ is still positive then let
 %%    $e_i\in \cE$ be a tight set with $v_i\in e_i$ (note that by the
 %%    initialization of $y$ we have $e_i\subseteq \{v_1,v_2,\dots,v_i\}$).

\untab

%\item EndFor

\item Let $i=1$, $S=\emptyset$ and $\cE'=\emptyset$.\hfill{//Throughout $S\subseteq V$ and $\cE'\subseteq \cE$}

\item For  $i=1,2,\dots, n$ do

\tab

  \item If $y_{v_i}>0$ and $v_i\not\in S$ then

    \tab
      \item Let $\cE':=\cE' + \{e_i\}$ and $S:=S\cup e_i$
    \untab

%    \item $i:=i+1$
    
\untab

%\item EndFor

\item Output $\cE'$ and $y$.

\end{pszkod}

The algorithm clearly proves the minmax theorem. In order to implement
it in polynomial time in $n=|V'|$ we have to provide the
subroutines needed in Steps \ref{st:maxw} and
\ref{st:miny}.
\end{proof}

\paragraph{Oracles for insolid sets}

Next, we analyze this algorithm for the solution of
Problem \ref{prob:maxsub} to show that it can be implemented in
polynomial time when $\cE$ is the family of insolid sets and $val
_e=k-\varrho _{D}(e)$ (for this weight function $y(e)\ge
val_e\ \forall e\in \cE$ implies $y(Z)\ge k-\varrho _{D}(Z)$ for every
non-empty $Z\subseteq V$). For that, we need to establish subroutines
for Steps \ref{st:maxw} and \ref{st:miny} that have running time
polynomial in the size of the graph.

It is easy to realize Step \ref{st:maxw} of Algorithm \maxal:
$M=k+1$ will be a good choice.

Step \ref{st:miny} can be realized with minimum cut computation as
follows. Let $V_i=V(T_{v_i})$ be the node set of the subtree of $T$
rooted at $v_i$ and define a  digraph $D'=(V+s,A')$
as follows: add a new node $s$ to $D$, introduce an arc $sv$ from $s$
to every $v\in V_i$ of multiplicity $y_v$ (note that $y$ stays integer during the algorithm).
%, and for an arc $uv\in A(D)$ we define 
%$w'(uv):=w(uv)$. 
%, and contract $s+(V-V_i)$. 
Then by a minimum cut algorithm we find an inclusionwise minimal
minimizer $Z$ of $\min\{\varrho_{D'}(Z): v_i\in Z\subseteq V_i\}$.  
%We can
%solve this problem based on a minimum cut computation, that is,
%finding an inclusionwise minimal set $Z$ that minimizes $\{\varrho_{D'}(X): v_i\in X\subseteq V_i\}$.
%minimum cut between $v_i$ and
%$\{s\}\cup (V-V_i)$.  
We claim that if $y(Z-v_i)< k-\varrho(Z)$ (that is, $y(v_i)$ will not
decrease to 0) then $Z$ is in-solid in $D$. By Claim \ref{claim4},
there is an insolid ($D$-insolid, to be precise) set $Z'\subseteq Z$
such that $\varrho_{D}(Z')\le \varrho_{D}(Z)$.  If $v_i\in Z'$ then
$Z'=Z$ by our choice of $Z$.  If $v_i\notin Z'$, then $y(Z')
+\varrho_D(Z')-k \le y(Z-v_i) + \varrho_D(Z) -k < 0$, contradicting
with $y(Z') \ge k- \varrho_D(Z')$.  In other words, we in fact
maintain $y(Z)\ge k-\varrho_D(Z)$ for every non-empty $Z\subseteq V'$
(besides non-negativity of $y$), and in Step \ref{st:miny} the set
$Z=e_i$ we find is insolid only if $y(v_i)$ does not decrease to zero.

This finally proves that Problem \ref{prob:maxsub} can be solved in
polynomial time.

\subsection{Enforcing the size requirement}

For an arbitrary  digraph $D=(V,A)$, nonempty subset
$V'\subseteq V$ and positive integer $k$, recall that $\bestsub(V')$
denotes an optimum solution of our unconstrained Problem
\ref{prob:maxsub} (that is, a maximizer of $\max\{\sum_{X\in
  \cX}(k-\varrho(X)): \cX$ is a subpartition of $V'\}$). This can be
found in polynomial time by the previous section: for a set $V'$ that
induces a subtree of $\tin$ we have $\bestsub(V')$ directly in the
output of $\maxal(V',\tin[V'],\cFin[V'], k-\varrho_{D})$, where $\cFin[V']$ denotes the family of insolid sets contained in $V'$.

%\fixme{With some extra effort this enforcing could also be described for an
%  arbitrary subtree family.}

\begin{pszkod}{Algorithm \bestconsub($D,k$)}

\item[] INPUT: A digraph $D=(V,A)$ and a positive integer $k$.

\item[] OUTPUT: A maximizer of $\max\{\sum_{X\in \cX}(k-\varrho(X)):
  |\cX|\ge 2, \cX$ is a subpartition of $V\}$

\item If the subpartition \bestsub($V$) has at least 2 members then
  output this and STOP.

\item Let $T= \tin$ be a basic tree of $D$. \label{st:Tin}

\item For every edge $e$ of $T$ consider at most 3 candidates defined as follows.

\tab

\item Let $V_1$ and $V_2$ be the node sets of the 2 components of $T-e$.

\item For $i=1,2$, let $X_i$ be an inclusionwise minimal minimizer of
  $\min\{\varrho_{D}(X): \emptyset\ne X\subseteq V_i\}$ (which can be
  found with minimum cut computations).

\item For both $i=1,2$, let $\cP_i=$\bestsub($V_i$).\label{st:bestsub}

\item Candidate 0 is $\{X_1,X_2\}$.

\item  If $|\cP_1|\ge 1$ then Candidate 1 is $\cP_1\cup \{X_2\}$.

\item  If $|\cP_2|\ge 1$ then Candidate 2 is $\cP_2\cup \{X_1\}$.

\untab

\item Output the best of the candidates above.

\end{pszkod}

\begin{thm}
Algorithm \bestconsub\ is correct.
\end{thm}
\begin{proof}
Clearly, the Algorithm outputs a feasible solution (that is, a
subpartition with at least 2 members), since all its candidates are
feasible. Let $\cP$ be an optimal insolid subpartition. The proof is
complete if we show that one of the candidates is at least as good as
\cP.  %Let $\tin$ be a basic tree for the insolid sets of $D_w$ and
Choose an edge $e$ of \tin\ so that one of the components (call it $V_1$) of
$T-e$ contains only one member of $\cP$. If \cP\ has only 2 members
then clearly Candidate 0 for the edge $e$ is not worse than \cP, so we
can assume that $|\cP|\ge 3$.  Let $V_2$ be the other component of
$T-e$. If \bestsub($V_2$) has at least 1 member then Candidate 2 is
not worse than \cP. But if \bestsub($V_2$) has no members then Claim
\ref{cl:optempty} gives that $\varrho_{D_w}(X)\ge k$ for every
nonempty $X\subseteq V_2$, therefore Candidate 0 should not be worse
than \cP, finishing the proof.
\end{proof}

The following example shows that even if \bestsub($V$) is not feasible
(that is, it has only 1 member), \bestconsub($V$) might have more than 2
members: let $k=4$ and the digraph be directed circuit of size 3,
where every arc has multiplicity 3.

\subsection{Analysis}\label{sec:anal}

In the algorithm designed above to solve Problem \ref{prob:maxsub2},
we apply two steps. In the first step (Step \ref{st:Tin}) we determine
a basic tree representation of the in-solid sets, which according
to \cite{bbf} can be done in $n^3 S(n,m)$ time, where $n$ is the
number of nodes and m is the number of edges in $D$, and $S(n,m)$
denotes the time complexity of finding a minimum $s-t$-cut in a
digraph with $n$ nodes and $m$ arcs. In the second step
(Step \ref{st:bestsub}), we apply our algorithm for
Problem \ref{prob:maxsub} for $O(n)$ different subsets $U$ of $V$,
which are determined by removing an edge from the basic tree.  For any
given subset $U$, Problem \ref{prob:maxsub} is solved in time $n
S(n+1,3m)$, since we apply a minimum cut algorithm to determine the
value of $y(v)$ for all nodes $v$ in $U$ by a minimum cut computation
for graphs of at most $n+1$ nodes and at most $m+kn$ arcs. (To see
this, note that only one node s is added to the graph, and the number
of arcs added is equal to the sum of $y(v)$'s, which is equal to the
sum $\sum _{X\in \pi}(k-\varrho ( X))$ over a partition $\pi$ of
$U+s$.) Note that $kn \le 2m$, since otherwise there may be no $k$
disjoint arborescences. The total running time for the algorithm
amounts to $n^2 S(n+1,3m)$. Thus, by using Orlin's
algorithm \cite{Orlin} for minimum cut computations, the running time
is bounded by $O(n^4m)$ (the bottleneck being Step  \ref{st:Tin}).

\section{The weighted case}\label{sec:weight}

In this section we solve Problems \ref{prob:3} and \ref{prob:4}. Our
algorithms are only polynomial if $k$ is fixed (not part of the input).

\subsection{Weighted blocking of \krarb s}\label{sec:krarb}

First we give an algorithm solving Problem \ref{prob:4}. By Edmonds'
disjoint arborescence theorem, the optimum solution will be all but
$k-1$ arcs entering a nonempty subset $X\subseteq V-r$. Thus what we
do is that we guess which $k-1$ arcs remain.

\newcommand{\blockkarb}{\textsc{Blocking-\karb{s}}}
\newcommand{\blockkrarb}{\textsc{Blocking-\krarb s}}

\begin{pszkod}{Algorithm \blockkrarb}
\item[] INPUT: A digraph $D=(V,A)$, a weight function $w:A\to \Rset_+$, a node $r\in V$,
  and a positive integer $k$. (We assume that $\varrho(r)=0$.)

\item[] OUTPUT: A subset $H$ of the arcs so that there is no \krarb\ in $D-H$ and $w(H)$ is minimum.

\item Let $best=\infty$.

\item For any subset $E$ of $A$ with $|E|=k-1$ do

\tab

  \item Find a minimizer $X_0$ of   $\min\{\varrho_{D_w-E}(X):\emptyset\ne X\subseteq V-r\}$.

  \item If $\varrho_{D_w-E}(X_0)<best$ then let $best = \varrho_{D_w-E}(X_0)$ and $H = \delta^{in}_{D-E}(X_0)$

\untab

\item Output $H$.

\end{pszkod}

This algorithm runs in time $O(m^kHO(n,m))$ time where $HO(n,m)$ denotes the time complexity of determining $\min\{\varrho_{D_w}(X):\emptyset\ne X\subseteq V-r\}$ in an arc-weighted digraph $D_w$ with $n$ nodes and $m$ arcs. Using the algorithm of Hao and Orlin \cite{HaoOrlin} we have $HO(n,m)=O(nm\log(n^2/m))$, giving that Algorithm \blockkrarb\ has running time $O(m^knm\log(n^2/m))$.

\subsection{Weighted blocking of  \karb s}\label{sec:karb}

%\fixme{What if $D$ has no \karb\ at all.}
Now we turn to  Problem \ref{prob:3}. 

\begin{cl}\label{cl:prob3}
Given a digraph $D=(V, A)$ that contains a \karb, let $H\subseteq A$
such that $D-H$ does not contain a \karb, and $H$ is inclusionwise
minimal to this property. Then there exists a subpartition $\cX$ of
$V$ such that $2\le |\cX|\le k+1$ and $\sum_{X\in \cX}\varrho_{D-H}(X)=
k(|\cX|-1)-1$.
\end{cl}

\begin{proof}
By Theorem \ref{thm:Frank}, there exists a subpartition \cX\ of $V$
such that $\sum_{X\in \cX}\varrho_{D-H}(X)< k(|\cX|-1)$: choose such
an \cX\ with $|\cX|$ smallest possible.  By this minimal choice of
$\cX$, $\varrho_{D-H}(X)< k$ for every $X\in \cX$ (if
$\varrho_{D-H}(X)\ge k$ for some $X\in \cX$ then $\cX-\{X\}$ would just as
well do). If $|\cX|>k+1$ then let $\cX'\subseteq \cX$ be arbitrary with
$|\cX'|=k+1$ and observe that $\sum_{X\in \cX'}\varrho_{D-H}(X)\le
(k+1)(k-1)<k^2$, contradicting the minimal choice of $\cX$. Therefore
$|\cX|\le k+1$ indeed holds ($|\cX|\ge 2$ is straightforward). 
By the minimality of $H$,
$H\subseteq \cup_{X\in \cX}\delta^{in}_D(X)$ and
$\sum_{X\in \cX}\varrho_{D-H}(X)= k(|\cX|-1)-1$.
\end{proof}

\newcommand{\opt}{\ensuremath{opt}}
\newcommand{\OPT}{\ensuremath{OPT}}
\newcommand{\alg}{\ensuremath{alg}}

By Claim \ref{cl:prob3}, the optimum solution  of
Problem \ref{prob:3} will be all but $k(|\cX|-1)-1$ arcs from a set
$\cup_{X\in \cX}\delta^{in}(X)$ for some subpartition \cX\ with $2\le
|\cX|\le k+1$. Our algorithm below guesses this set of remaining arcs.
As a subroutine we need an algorithm solving the following problem.

\begin{prob}\label{prob:minsub}
Given a digraph $D=(V,A)$, a weight function
$w:A\to \Rset_+$ and a positive integer $t$, determine
$\min\{\sum_{X\in \cX}\varrho_{D_w}(X): \cX$ is a subpartition of $V$
and $|\cX|=t\}$.
\end{prob}

This problem will be solved in Section \ref{sec:aux}.

\newcommand{\jel}{\ensuremath{candidate}}

\begin{pszkod}{Algorithm \blockkarb}

\item[] INPUT: A digraph $D=(V,A)$, a weight function $w:A\to \Rset_+$,  and a positive integer $k$. 

\item[] OUTPUT: $\min\{w(H): $ there is no \karb\ in $D-H\}$.

%A subset $H$ of the arcs so that there is no \karb\ in $D-H$ and $w(H)$ is minimum.

\item If there is no \karb\ in $D$ then output 0 and STOP.

\item Let $best=\infty$.

\item For $t=2, 3, \dots, k+1$ do

\tab
 
  \item For every $E\subseteq A$ of size $k(t-1)-1$ do

  \tab
 
    \item \label{st:aux} Let $\jel
    = \min\{\sum_{X\in \cX}\varrho_{D_w-E}(X): \cX$ is a subpartition
    of $V$ and $|\cX|=t\}$.

    \item If $\jel < best$ then $best = \jel$.

  \untab
  
\untab

\item Output $best$.

\end{pszkod}

For sake of simplicity we formulated Algorithm {\blockkarb} so
that it outputs the weight of the optimal arc set that blocks
all \karb s: the algorithm can be obviously modified to return the
optimal arc set instead. 
The running time 
of Algorithm \blockkarb\ 
will be analyzed in
Section \ref{sec:runtime}.
The proof of correctness is as follows.
% of Algorithm \blockkarb\ is left to the reader.
%% \begin{thm}
%% Algorithm \blockkarb\ is correct.
%% \end{thm}
%% \begin{proof}
Let \alg\ be the output of Algorithm \blockkarb\ and $\opt=w(\OPT)$ be
the optimum solution. Clearly, $\opt\le \alg$. On the other hand, by
Claim \ref{cl:prob3}, there exists a subpartition \cX\ such that
$\sum_{X\in \cX}\varrho_{D-\OPT}(X)= k(|\cX|-1)-1$ and $2\le |\cX|\le k+1$; for
$E= \cup_{X\in \cX}\delta^{in}_{D-\OPT}(X)$ the algorithm will find a candidate that
is not worse than \opt.
%\end{proof}

\subsubsection{Solution of Problem \ref{prob:minsub}}\label{sec:aux}

In this section we solve Problem \ref{prob:minsub} that is used in
Step \ref{st:aux} of Algorithm \blockkarb.  Clearly, we can assume
that the optimal soltution is an insolid subpartition.  Note the
difference between this problem and Problem \ref{prob:maxsub2}: here
we have exact restriction on the size of the subpartition to be found,
not just a lower bound.

\newcommand{\fixedpart}{\textsc{Best-Fixed-Subpart}}

\begin{pszkod}{Algorithm \fixedpart}
\item[] INPUT: A digraph $D=(V,A)$, a weight function $w:A\to \Rset_+$,  and a positive integer $t$. 

\item[] OUTPUT: $\min\{\sum_{X\in \cX}\varrho_{D_w}(X): \cX$ is a subpartition of $V$ and $|\cX|=t\}$.

\item Let $T$ be a representative tree for the insolid sets of the weighted digraph  $D_w$.

\item Let $best=\infty$

\item For every $F\subseteq E(T)$ of size $t-1$ do

\tab

        \item Let $Z_1, Z_2, \dots, Z_t$ be the node sets of the
        connected components of $T-F$.

        \item Let $\jel = 0$

        \item For $i=1, 2 \dots ,t$ do

        \tab
        
                \item $\jel \mathrel{+}= \min\{\varrho_{D_w}(X): \emptyset\ne
                X\subseteq Z_i\}$.

        \untab 

        \item If $\jel<best$ then $best = \jel$.
\untab

%\item EndFor

\item Output $best$.

\end{pszkod}

%The proof of correctness of Algorithm \fixedpart\ is left to the reader.

\begin{cl}
Algorithm \fixedpart\ returns a correct answer.
\end{cl}

\begin{proof}
Let \alg\ be the output of Algorithm \fixedpart\ and $\opt$ be
the optimum solution. Clearly, $\opt\le \alg$. On the other hand, we
can assume that $\opt = \sum_{X\in \cX}\varrho_{D_w}(X)$ for some
insolid subpartition $ \cX$ of $V$ with $|\cX|=t$. Since $T$ is a
representative tree of insolid sets of $D_w$, we can choose some
$F\subseteq E(T)$ with $|F|=t-1$ such that every component of $T-F$
contains exactly one member of \cX. For this particular $F$ the
algorithm will find a candidate not worse than \opt.
\end{proof}

\subsubsection{Running time}\label{sec:runtime}

In this section we analyze the running time of
Algorithm \blockkarb. Let $n$ and $m$ denote the number of nodes and
arcs of the input digraph $D$. Recall that $S(n,m)$ denotes the time
complexity of finding a minimum $s-t$-cut in a weighted digraph with
$n$ nodes and $m$ arcs, and similarly $HO(n,m)$ denotes the time
complexity of determining $\min\{\varrho_{D_w}(X): \emptyset \ne
X \subseteq V' \}$, where $V'\subseteq V$.  The running time of
Algorithm \fixedpart\ is $n^{t-1}HO(n,m)$ plus the time needed to
determine the representative tree $T$, which can be done in time
$O(n^3S(n,m))$, as mentioned in Section \ref{sec:anal}. Thus we get
$O(m^{k^2}(n^kHO(n,m)+n^3S(n,m)))$ as an overall complexity for
Algorithm \blockkarb. Substituting $S(n,m)=O(nm)$ (\cite{Orlin}) and
$HO(n,m)=O(nm\log(n^2/m))$ (\cite{HaoOrlin}) we get
$O(m^{k^2}(n^{k+1}m\log(n^2/m)+n^4m))$ for the running time.

%% We have worked out an algorithm for Problem \ref{prob:3} that runs in
%% $m^{O(k^2)}$: we will describe the details in the proceedings version
%% of this paper.

\section{Acknowledgements}

The authors wish to thank Krist\'of B\'erczi and Tam\'as Kir\'aly for
observing the flaw in a previous version and for useful discussions on
the topic.

\bibliographystyle{amsplain} \bibliography{bkmincost}

\end{document}